\documentclass[11pt,final]{article}
\usepackage{ifthen}

\newboolean{PSdiagram}

\setboolean{PSdiagram}{true} %

\usepackage{mathrsfs}
\usepackage{amssymb}
\usepackage{amsmath}
\usepackage{dsfont}
\usepackage[heavycircles]{stmaryrd}

\usepackage[latin9]{inputenc}
\usepackage[english]{babel}

\usepackage{theorem}
\theoremstyle{change}

\ifthenelse{\boolean{PSdiagram}} %
{
  \usepackage{diagrams}
  \diagramstyle[height=0.27in,width=0.35in,labelstyle=\scriptstyle,dpi=600,nohug,PostScript=dvips]
}{
  \usepackage[UglyObsolete]{diagrams}
  \diagramstyle[height=0.27in,width=0.35in,labelstyle=\scriptstyle,dpi=600,nohug,noPostScript]
}

\renewcommand\textrm[1]{{\rm #1}}

\newlength{\listindent}
\newcommand\enumroman[1]{{\rm(\/\textit{\roman{#1}}\/)}}
\newcommand\enumarabic[1]{\arabic{#1}.}
\newcounter{enumcounter}
\newcommand\enum[1]{\setcounter{enumcounter}{#1}\enumroman{enumcounter}}
\newcommand\impl[2]{\enum{#1} $\then$ \enum{#2}}

\newenvironment{enumprop}
   {\begin{list}{\enumroman{enumiv}}{\usecounter{enumiv}
                      \setlength{\topsep}{0.1ex plus0.1ex minus0.1ex}
                      \setlength{\itemsep}{0.1ex plus0.1ex minus0.1ex}
                      \setlength{\parsep}{0.1ex plus0.1ex minus0.1ex}
                      \setlength{\itemindent}{0in}
                      \setlength{\labelwidth}{2.5em}
                      \setlength{\leftmargin}{2\listindent}
                      \setlength{\partopsep}{\medskipamount} } }
   {\end{list}}

\newenvironment{enumsrm}
   {\begin{list}{\enumroman{enumiv}}{\usecounter{enumiv}
                      \setlength{\topsep}{0.1ex plus0.1ex minus0.1ex}
                      \setlength{\itemsep}{0.1ex plus0.1ex minus0.1ex}
                      \setlength{\parsep}{0.1ex plus0.1ex minus0.1ex}
                      \setlength{\itemindent}{0in}
                      \setlength{\labelwidth}{2.5em}
                      \setlength{\leftmargin}{2\listindent}
                      \setlength{\partopsep}{\medskipamount} } }
   {\end{list}}

\newcommand\varhspace[1][1em]{\hspace{0em plus#1 minus0.3em}}

\newarrowtail{<=}\Leftarrow\Rightarrow\Uparrow\Downarrow

\newarrow{Corresponds} <--->
\newarrow{Equal} =====
\newarrow{Equiv} {<=}==={=>}
\newarrow{Dash} {}{dash}{}{dash}{}
\newarrow{Empty} {}{}{}{}{}
\newarrow{Dashinto} C{dash}{}{dash}>
\newarrow{Dashonto} {}{dash}{}{dash}{>>}
\newarrow{Dotmapsto} {mapsto}...>

\newarrow{CTo} ----{->}
\newarrow{CMapsto} |---{->}
\newarrow{CInto} C---{->}
\newarrow{COnto} ----{->>}

\def\nepok{\overprint{\raise1.5ex\rlap{$\;\;\urcorner$}}}
\def\nwpok{\overprint{\raise2.5ex\llap{$\lrcorner\;\;$}}}
\def\sepbk{\overprint{\lower1.5ex\rlap{$\;\;\lrcorner$}}}

\newcommand\diagramiso{\hbox{\lower1.5ex\hbox{$_\sim$}}}

\theorembodyfont{\it}

\newtheorem{defn}{Definition}[section]

\newtheorem{lemma}[defn]{Lemma}
\newtheorem{prop}[defn]{Proposition}
\newtheorem{thm}[defn]{Theorem}
\newtheorem{cor}[defn]{Corollary}

\newenvironment{interlude}{\begin{trivlist}\setlength\topsep{\theorempreskipamount}\item}{\end{trivlist}}

\newenvironment{remarks*}{\begin{interlude}\pagebreak[2]\noindent\textbf{Remarks.\ }\nopagebreak}{\end{interlude}}

\newenvironment{examples*}{\begin{interlude}\pagebreak[2]\noindent\textbf{Examples.\ }\nopagebreak}{\end{interlude}}

\newenvironment{proof}{\begin{interlude}\noindent\textit{Proof.}}{\qed\end{interlude}}

\newenvironment{proof*}{\begin{interlude}\noindent\noindent\textit{Proof.}}{\end{interlude}}

\newcommand\qedsymbol{$\square$}
\newcommand\qed{\nopagebreak\penalty10000\hbox{}\nobreak\penalty10000\nopagebreak\penalty10000\hfill\hbox{\qedsymbol}\par}

\def\oldphi{\mathchar"11E}
\def\phi{\varphi}
\def\theta{\vartheta}
\def\epsilon{\varepsilon}
\def\sub{\subseteq}

\def\mathbb{\mathds}

\newcommand\bbN{{\mathbb{N}}}
\newcommand\bbZ{{\mathbb{Z}}}
\newcommand\bbQ{{\mathbb{Q}}}

\newcommand\bbP{{\mathbb{P}}}

\newcommand\bbG{{\mathbb{G}}}

\newcommand\bbQZ{\bbQ\mod\bbZ}

\DeclareMathOperator{\preim}{pre\kern0.13em im} %
\DeclareMathOperator{\preker}{pre\kern0.13em ker} %
\DeclareMathOperator{\precoker}{pre\kern0.13em coker} %
\DeclareMathOperator{\aut}{Aut} %
\renewcommand{\hom}{\operatorname{Hom}}
\DeclareMathOperator{\gal}{Gal} %
\DeclareMathOperator{\quot}{quot} %
\DeclareMathOperator{\ord}{ord} %
\DeclareMathOperator{\pic}{Pic} %

\newcommand\tr{\textrm{tr}}

\newcommand\gp{\textrm{\rm GP}}

\def\iso{\cong}

\newcommand\frakm{\mathfrak m}

\newcommand\defeq{:=}
\newcommand\fedeq{=:}
\newcommand\then{\Rightarrow}

\renewcommand\iff{\Leftrightarrow}

\newcommand\Iff{\;\Longleftrightarrow\;}

\newcommand\set[2][auto]{
     \ifthenelse{\equal{#1}{auto}}{\left\lbrace}{\csname #1\endcsname\lbrace} #2 \ifthenelse{\equal{#1}{auto}}{\right\rbrace}{\csname #1\endcsname\rbrace} }

\let\modulo\mod

\renewcommand\mod{/}

\let\tensor\otimes

\newcommand{\indlim}{\operatorname*{\underrightarrow{\lim}}} %
\renewcommand{\projlim}{\operatorname*{\underleftarrow{\lim}}}

\newcommand\longto{\longrightarrow}
\newcommand\isoto{\mbox{$\hspace{7.5pt}\raise 3pt\hbox{$\sim$}\hspace{-17pt}\longrightarrow\hspace{3pt}$}\linebreak[0]}
\newcommand\isoot{\mbox{$\hspace{8.5pt}\raise 3pt\hbox{$\sim$}\hspace{-18pt}\longleftarrow\hspace{3pt}$}\linebreak[0]}
\newcommand\into{\hookrightarrow}

\newcommand\Longto[1]{\stackrel {#1}\longto}

\newcommand\biext{\operatorname{Biext}}

\newcommand{\HH}{{\rm H}} %
\newcommand\scrE{\mathscr E}

\newcommand\scrG{\mathscr G}
\newcommand\scrH{\mathscr H}

\newcommand\scrP{\mathscr P}
\newcommand\scrR{\mathscr R}

\DeclareMathOperator\shom{{\scrH\kern-0.5ex\it o\kern-0.1ex m}} %
\DeclareMathOperator\sext{{\scrE\!\it xt}} %

\newcommand\inv{^{-\!1}}

\newcommand\units{^\ast}

\title{Reduction of Abelian Varieties\\ and Grothendieck's Pairing}
\author{Klaus Loerke}
\date{2009}

\begin{document}

\nonfrenchspacing

\renewcommand{\contentsname}{Inhalt}
\def\MakeUppercase#1{#1}

\maketitle

\begin{abstract}
We prove that abelian varieties of small dimension over discrete valuated, stricty henselian ground
fields with perfect residue class field obtain semistable reduction after a tamely ramified
extension of the ground field. Using this result we obtain perfectness results for Grothen\-dieck's
pairing.
\end{abstract}

\section{Introduction}

Let $R$ be a discrete valuation ring with field of fractions $K\defeq\quot R$  and with residue
class field $k\defeq R\mod\frakm$. Furthermore let $A_K$ be an abelian variety over $K$ with its
dual $A'_K$ and Néron models $A$ and $A'$, respectively. Their component groups are denoted by
$\oldphi$ and $\oldphi'$. The duality between $A_K$ and $A_K'$ is reflected by the \emph{Poincaré
bundle} $\scrP$ on $A_K\times_K A'_K$. It has the property that the induced maps
\[
A_K'\longto \pic^0 A_K,\quad a\mapsto\scrP|_{A_K\times\set a}
\]
and vice versa are isomorphisms. We wish to extend $\scrP$ to the level of the associated Néron
models, in order to study the relationship between $A$ and $A'$. An appropriate setting for this is the
notion of \emph{biextensions}. We briefly review the needed theory (cf.\ \cite{sga7}, VII).
\par
The canonical sequence
\[
0\longto\bbG_{m,R}\longto\scrG\longto i_\ast\bbZ\longto 0,
\]
where $\scrG$ denotes the Néron model of $\bbG_m$, gives rise to an exact sequence
\[
\biext^1(A_R,A'_R,\bbG_{m,R})\longto\biext^1(A_R,A'_R,\scrG)\longto\biext^1(A_R,A'_R,i_\ast\bbZ).
\]
After canonical identifications (\cite{Bo97}, section 4), this sequence can be written as
\[
\biext^1(A_R,A'_R,\bbG_{m,R})\longto\biext^1(A_K,A'_K,\bbG_{m,K})\longto\hom(\oldphi\tensor\oldphi',\bbQZ).
\]
We regard the bundle $\scrP$ as an element of the group $\biext^1(A_K,A'_K,\bbG_{m,K})$. Now,
Grothendieck's Pairing, $\gp$ for short, is defined to be the image of $\scrP$ in the group
$\hom(\oldphi\tensor\oldphi',\bbQZ)$. It represents the obstruction to extend $\scrP$ to a biextension
of $\bbG_{m,R}$ with $A$, $A'$. Grothendieck conjectured this pairing to be perfect. Indeed, for perfect residue
class field it can be shown that it is perfect\footnote{For mixed characteristic and perfect
residue class field cf.\ \cite{beg}, for finite $k$ cf.\ \cite{mcc86}. If $k$ is not perfect,
counterexamples can be found, \cite{bb}, corollary 2.5. } in all cases except of the case of a
discrete valuation ring of equal characteristic $p\neq 0$ with infinite residue class field $k$.

\begin{prop}
Let $A_K$ be an abelian variety.  Then the prime-to-$p$-part of Gro\-then\-dieck's pairing is
perfect. If $A_K$ has semistable reduction, the whole pairing is perfect.
\end{prop}

\begin{proof}
The first part is \cite{Be01}, Theorem 3.7, the second \cite{We}.
\end{proof}

Grothendieck's pairing induces a morphism $\oldphi'\to\hom(\oldphi,\bbQZ)\fedeq\oldphi^\ast$. It
fits into the following diagram:

\begin{prop}\label{gp_diag}
There is the following commutative diagram of sheaves with respect to the smooth topology
\begin{diagram}
0 & \rTo & A'^0              & \rTo & A'             & \rTo & i_\ast\oldphi' & \rTo & 0\\
  &      & \dInto            &      & \dEqual        &      & \dTo>\gp \\
0 & \rTo & \sext^1(A,\bbG_m) & \rTo & \sext(A,\scrG) & \rTo & \sext^1(A,i_\ast\bbZ)& = i_\ast\oldphi^\ast, %
\end{diagram}
in which the sheaves in the second line can be represented uniquely by smooth schemes. In particular, 
the sheaf $\sext^1(A,\bbG_m)$ can be represented by an open and closed subscheme of $A'$, whose 
components correspond to the elements of
$\ker\gp$. This is a $p$-group.
\end{prop}

\begin{proof}
For existence of the diagram see \cite{Bo97}, section 5, the last assertion is the snake lemma and
\cite{Be01}, Theorem 3.7.
\end{proof}

Let us denote by $\gp$ both the pairing
$\oldphi\times\oldphi'\to\bbQZ$ itself and the induced morphism $\oldphi'\to\oldphi^\ast$ (or
$\oldphi\to(\oldphi')^\ast$, respectively). Since $\oldphi$ and $\oldphi'$ are finite groups, the
induced morphisms are bijective, i.\,e.\ Grothendieck's pairing is perfect, if and only if the
induced morphisms are injective. Consequently, $\gp$ is perfect, if and only if the schemes
representing $\sext^1(A,\bbG_m)$ and $\sext^1(A',\bbG_m)$ are connected.

\section{Weil Restriction}

As before, let $R$ be a discrete valuation ring, $K\defeq\quot R$ its field of fractions and let
$A_K$ be an abelian variety over $K$ with dual $A'_K$ and N\'eron models $A_R$ and $A'_R$ over $R$.
It
is known that there exists a Galois extension $L/K$ such that $A_L\defeq A_K\tensor_K L$ has
semistable reduction (\cite{sga7}, IX, 3.6). In this case, Grothendieck's pairing for $A_L$
and $A'_L$ is perfect (\cite{We}). We want to examine the relationship between Grothendieck's
pairing for $A_K$ and $A_L$:
Let $S$ be the integral closure of $R$ in $L$.
It induces a residue class field extension $\ell/k$. The N\'eron models of $A_L$ and $A'_L$ will be
denoted by $A_S$ and $A'_{S}$.  The pairings of groups of components of the N\'eron models over $R$
and $S$ can be summarised by the following commutative diagram of $\gal(k_s/\ell)$-modules.
\begin{diagram}[width=0.25in,l>=0.35in] 
\oldphi_{A_R} & \times & \oldphi_{A'_R} & \rTo & \bbQZ \\
\dTo          &        & \dTo           &      & \dTo>e    \\
\oldphi_{A_S} & \times & \oldphi_{A'_S} & \rTo & \bbQZ,
\end{diagram}
where $e$ is the ramification index\index{ramification index} of $L/K$. (\cite{sga7}, XVII, 7.3.5).
Unfortunately, we cannot conclude that the first pairing is perfect if the second one is. However,
we will use the technique of \emph{Weil restriction}\index{Weil restriction} to infer a partial
result. All Weil restrictions we will encounter are representable by smooth group schemes,
\cite{blr}, 7.6, theorem 4. Let
\[
X_K\defeq\scrR_{L/K} A_L \qquad\text{and}\qquad X_R\defeq\scrR_{S/R} A_S
\]
be the Weil restriction of the abelian variety $A_L$ and its N\'eron model $A_S$.

\begin{prop}
In this situation, $X_R$ is the N\'eron model of $X_K$.
\end{prop}

\begin{proof}
Let $T$ be a smooth $R$-scheme. The calculation
\begin{align*}
\hom_R(T,X_R) & =\hom_S(T\tensor_R S,A_S) \\
              & =\hom_L(T\tensor_R K\tensor_K L,A_L) \\
              & =\hom_K(T\tensor_R K,X_K)
\end{align*}
shows that $X_R$ has the universal property of the N\'eron model of $X_K$.
\end{proof}

By $\oldphi_{X_R}$, $\oldphi_{A_S}$ etc., we denote the corresponding groups of components. In this
situation, we have the following proposition:

\begin{prop}\label{weil_phi_iso}
The canonical morphism $X_R\tensor_R S\to A_S$ induces a morphism
\[
\oldphi_{X_R}\tensor_k \ell\longto\oldphi_{A_S}.
\]
If $\ell/k$ is purely inseparable, this is an isomorphism.
\end{prop}

\begin{proof}
\cite{bb}, proposition 1.1.
\end{proof}

From now on, we assume $R$ to be a strictly henselian discrete valuation ring with perfect, i.\,e.\
algebraically closed residue class field of characteristic $p\neq 0$. In this case, every finite
extension
of $k$ is trivial. In particular, it is purely inseparable and we can identify $\oldphi_{A_S}$ and
$\oldphi_{X_R}$ by means of this proposition.
\par
In this situation, we can allow a tamely ramified extension of $K$ to test whether or not
Grothendieck's pairing is perfect:

\begin{prop}\label{gp_perf_tamely}
Let $R$ be as stated above and let $L/K$ be a tamely ramified Galois extension. \index{tamely
ramified extension}\index{Grothendieck's pairing} Consider the following assertions:
\begin{enumprop}
  \item Grothendieck's pairing for $A_S$ and $A'_S$ is perfect.
  \item Grothendieck's pairing for $X_R$ and $X'_R$ is perfect.
  \item Grothendieck's pairing for $A_R$ and $A'_R$ is perfect.
\end{enumprop}
Then we have $\enum 1\iff\enum 2\then\enum 3$.
\end{prop}

We sketch the proof as given in \cite{bb}, lemma 2.2 and corollary 3.1:

\begin{proof}
Let $n\defeq[L:K]$. Since $k$ is algebraically closed, we have $n=e_{L/K}$. Since $L/K$ is tamely
 ramified, $p$ does not divide $n$. As shown in \cite{bb}, Lemma 2.2 we have the
following equality:
\[
[e_{L/K}]\circ \gp_{X_R}=[n]\circ\gp_{A_S},
\]
where $[n]$ denotes the $n$-multiplication on the appropriate group of components. It follows that
the kernels of both compositions coincide (after identification as in proposition
\ref{weil_phi_iso}). Since $\gp$ is an isomorphism on the prime-to-$p$-part and since $n$ is prime
to $p$, the kernels of $\gp_{X_R}$ and $\gp_{A_S}$ coincide. This shows the equivalence of \enum 1
and \enum 2.
\par
Consider the canonical morphisms $A_K\into X_K$ and the norm map $X_K\to A_K$. Their composition is
multiplication with $n$, cf.\ the proof of \cite{bb}, corollary 3.1. These morphisms give rise to
the following morphisms of N\'eron models:
\[
A_R\longto X_R\longto A_R,
\]
such that their compositions is the multiplication with $n$. Hence, there are morphisms of the
smooth schemes which represent $\sext^1(-,\bbG_m)$ and of their component groups. Since $\ker \gp$
is isomorphic to the group of components of the smooth scheme which represents
$\sext^1(A_R,\bbG_m)$, proposition \ref{gp_diag}, we have two morphisms
\[
\ker\gp_{A_R}\longto  \ker\gp_{X_R}\longto \ker\gp_{A_R},
\]
such that the composition is  multiplication by $n$. As these groups are $p$-groups by theorem
\ref{gp_diag} and $n$ is prime to $p$, the first morphism is injective; hence, \impl 2 3.
\end{proof}

\section{Abelian Varieties of Small Dimension}

We have seen that  perfectness of Grothendieck's pairing can be tested after a tamely ramified
extension of the ground field. We would like to derive a property of abelian varieties which acquire
semistable reduction\index{semistable reduction} after a tamely ramified extension $L/K$.
\index{tamely ramified extension} As Grothendieck's pairing of $A_L$ is perfect, we can conclude
that Grothendieck's pairing of $A_K$ is perfect, too. We will prove that abelian varieties of
small dimension, depending on the residue class field characteristic $p$, achieve
semistable reduction after a tamely ramified extension. This provides another new clue
that Grothendieck's conjecture is true if the residue class field $k$ is perfect. As
before, let $R$ be a strictly henselian, discrete valuation ring with algebraically closed residue
class field. In this case, the integral closure $S$ of $R$ in $L$ is a strictly henselian discrete
valuation ring with algebraically closed residue class field. In particular, the inertia
subgroups\index{inertia subgroup} of $\gal(K_s/K)$ and $\gal(K_s/L)$ for a fixed separable closure
$K_s$ of $K$ coincide with the absolute Galois groups of $K$ and $L$.

\begin{defn}[Tate module]\index{Tate module}
Let $A_K$ be an abelian variety and let $\ell\neq p$ be a prime. The \emph{Tate module} of $A_K$ is
the $\gal(K_s/K)$-module
\[
T_\ell(A_K)\defeq\projlim_{n\in\bbN} A_{K,\ell^n}(K_s).
\]
\end{defn}

As an abelian group, the Tate module is isomorphic to $\hat\bbZ_\ell^{2g}$, where $g$ is the
dimension of $A_K$. 
\par
Let $G\defeq\gal(K_s/K)$ denote the absolute Galois group of $K$ and let $I\sub G$ denote the
inertia subgroup. As we have seen, they coincide. Nonetheless, we use this notation for a coherent
statement of the following theorems. Our point of departure is the \emph{Galois criterion for
semistable reduction} (cf.\ \cite{sga7}, IX, 3.5):

\begin{prop}[Galois Criterion for Semistable Reduction]\label{galois_crit}\index{Galois
criterion}\index{semistable reduction}
\varhspace Let $A_K$ be an abelian variety over $K$ and let $\ell\neq p$ be a prime. Then the
following
statements are equivalent:
\begin{enumprop}
  \item $A_K$ has semistable reduction.
  \item There exists an $I$-submodule $T'\sub T\defeq T_\ell(A_K)$, such that $I$ operates trivially
   on $T'$ and $T\mod T'$.
\qed
\end{enumprop}
\end{prop}

As the inertia subgroup $I\sub G$ coincides with the whole Galois group $G$, we will not distinguish
between them.  Using the Galois criterion, we can formulate the semistable reduction theorem in
terms of Galois  theory:

\begin{prop}\label{ssred}
Let $A_K$ be an abelian variety.  There exists a normal subgroup $G'\sub G$ of finite index
with the property \enum 2, i.\,e.\ there exists a subgroup $T'\sub T$, stable under the action of
$G'$, such that $G'$ operates trivially on $T'$ and $T/T'$.
\end{prop}

\begin{proof}
It is known  that $A_K$ acquires semistable reduction after a finite
Galois extension $L/K$, corresponding to a normal subgroup $G'\defeq\gal(K_s/L)\sub G$ of finite
index, cf.\ \cite{sga7}, IX, 3.6. Since $R$ is strictly henselian with algebraically closed residue
class field, the integral closure $S$ of $R$ in $L$ is strictly henselian, again. Therefore, the
inertia  subgroup of $G'$ coincides with $G'$ and we can apply the Galois criterion.
\end{proof}

Our strategy is to enlarge $G'$ by an appropriate pro-$p$-group, such that the resulting field
extension is tamely ramified. Since $R$ is strictly henselian, the theory of tamely ramified extension
of $K$ reduces to the following:

\begin{prop}
Let $R$ be a strictly henselian discrete valuation ring with field of fractions $K$ and residue
class field $k$ of characteristic $p\neq 0$.
\begin{enumprop}
  \item A finite extension $L/K$ is tamely ramified if and only if $p\nmid[L:K]$. Each such
   extension   is a cyclic Galois extension.
  \item Let $L/K$ be any finite Galois extension with corresponding Galois group $G=\gal(L/K)$. Then there 
  exists a unique $p$-Sylow-subgroup $G_p\sub G$.
\end{enumprop}\index{wildly ramified extension}\index{tamely ramified extension}
Consequently, any Galois extension $L/K$ with Galois group $G$ can be splitted up into a tamely
ramified
Galois extensions $L^\tr/K$ with cyclic Galois group and a wildly ramified Galois extension
$L/L^\tr$ with Galois group $G_p$, where $G_p\sub G$ is the unique $p$-Sylow-subgroup
of $G$.
\end{prop}

\begin{proof}
Since $k$ is separably closed, every finite extension  $\ell/k$ is purely inseparable and, hence,
its degree is a power of  $p$. Due to the fundamental equation $[L:K]=e_{L/K}\cdot[\ell/k]$ the
first part of \enum 1 is obvious.
\par
Since $R$ is henselian, we can lift the roots of the polynomial $X^n-1$ from $k$ to $R$. Thus, we
can conclude that the $n$-th roots of unity are contained in $R$ and hence in $K$. Any extension
$L/K$ of degree $n$ prime to $p$ can easily be shown to be a Kummer extension\index{Kummer
extension}, isomorphic to
\[
L\iso K[X]\mod (X^n-\pi)
\]
for a suitable uniformising element $\pi$ of $K$. Now, $\sigma\mapsto \frac {\sigma(X)} X$
constitutes an isomorphism $\gal(L/K)\isoto\mu_n(K)$.  Since the $n$-th roots of unity are
contained in $R$ and thus in $K$, the group $\mu_n(K)$ is (non canonically) isomorphic to $\bbZ\mod
n$. This settles assertion \enum 1. To show \enum 2, let $G_p$ be any $p$-Sylow subgroup of $G$. The
degree of the corresponding field extension $L^{G_p}/K$ is prime to $p$. Hence it is a Galois
extension by \enum 1, which implies that $G_p$ is normal. 
\end{proof}

We set $K_s^\tr$ for the union of all tamely ramified extensions of $K$ in $K_s$. This field is
tamely ramified over $K$ and the above proposition can be generalized to the situation of profinite Galois
groups as follows:

\begin{cor}
Let $R$ be as above. Then there exists an exact sequence
\[
0\longto P\longto\gal(K_s/K)\longto\prod_{\ell\neq p}\hat\bbZ_\ell\longto 0
\]
with a pro-$p$-group $P=\gal(K_s/K_s^\tr)$. The last term is isomorphic to the Galois group $\gal(K_s^\tr/K)$.
\qed
\end{cor}

Let $L/K$ be a finite Galois extension with Galois group $G'$. We have seen that it gives rise to
normal field extensions $K\sub L^\tr\sub L$. We now want to describe the Galois extension
$K_s/L^\tr$ and its Galois group:

\begin{prop}\label{gal_Ltr}
In this situation, the field $L^\tr$ can be written as $L\cap K_s^\tr$. Thus we have $\gal(K_s/L^\tr)=G'\cdot P$ for the
pro-$p$-group $P=\gal(K_s/K_s^\tr)$.
\end{prop}

\begin{proof}
Since $K_s^\tr$ is the union of all tamely ramified extensions of $K$ in $K_s$, it is a tamely ramified extension of $K$.
Thus $K_s^\tr\cap L$ is the maximal tamely ramified extension of $K$ in $L$.
\end{proof}


\par
In the light of this proposition, we are to study the action of the group $P$ on the Tate module.
Since the Galois group acts on each of the groups $A_{K,\ell^n}(K_s)=(\bbZ/\ell^n)^{2g}$, it
suffices to study the action of $G$ and $P$ on these groups to understand the action on the entire
Tate module.
\par
An action of $G$ on these groups can be regarded as a homomorphism
\[
G\longto\aut((\bbZ/\ell^n)^{2g}).
\]
In a first step, we investigate the order of this automorphism group.

\begin{lemma}\label{autord}
Let $\ell$ be a prime and let $Z=\bbZ\mod\ell^{n_1}\times\ldots\times\bbZ\mod\ell^{n_r}$.
\begin{enumprop}
  \item $\ord\aut Z= \ell^\eta\cdot \prod_{i=1}^r(\ell^{d_i}-\ell^{i-1})$ for some
$d_i\in\set{i,\ldots,r}$ and $\eta\in\bbN$.
  \item $\ord\aut((\bbZ\mod\ell^n)^r)=\ell^{\eta}\cdot\prod_{i=1}^r (\ell^i-1)$ for some integer
$\eta\in\bbN$.
  \item If $H\sub(\bbZ\mod\ell^n)^r$ is a subgroup, then each prime divisor of $\ord\aut H$ is a
prime divisor of $\ord\aut((\bbZ\mod\ell^n)^r)$.
\end{enumprop}
\end{lemma}

\begin{proof}
The order of the automorphism group of a finite abelian $\ell$-group is computed in \cite{ra},
Theorem 15. For a modern account, see \cite{hr}, Theorem 4.1.
\par
If $H$ is a subgroup as in \enum 3, then it is isomorphic to
\[
H\iso\prod_{i=1}^s\bbZ\mod\ell^{n_i}\quad\text{for some $s\leq r$ and $n_i\leq n$},
\]
and we can write each factor $\ell^{d_i}-\ell^{i-1}$ of $\ord\aut H$ as
$\ell^{i-1}(\ell^{d_i-i+1}-1)$ for some $d_i\in\set{i,\ldots,s}$. Now assertion \enum 3 follows.
\end{proof}

The existence of elements of order $p$ in a finite group $G$ is equivalent to $p\mid\ord G$.
Therefore, it is sufficient to show that the product of factors $\ell^i-1$ is not divisible by $p$
in order to prove that the automorphism group in mind has no elements of order
$p$. Furthermore, if $\ord\aut G$ is not divisible by $p$, the same is true for the order of the
automorphism group of every subgroup of $G$.
\par
Let $\bbP$ denote the set of primes, $\bbP=\set{2,3,5,\ldots}$. Then we can formulate Dirichlet's
prime number theorem\index{Dirichlet's prime number theorem} as follows:

\begin{lemma}\label{dpnt}
Let $p$ be a prime. Then the canonical map $\bbP\setminus\set p\to(\bbZ\mod p)\units$ is
surjective.\qed
\end{lemma}

We are now ready to prove the following reduction theorem:

\begin{thm}\label{smalldim_ssred}\index{tamely ramified extension}\index{semistable reduction}
Let $R$ be a strictly henselian discrete valuation ring with residue class field of characteristic
$p\neq 0$. Let $K$ be the field of fractions of $R$, and let $A_K$ be an abelian variety over $K$
of dimension $g$. If $2g+3\leq p$, then $A_K$ obtains semistable reduction over a tamely ramified
extension of $K$.
\end{thm}

\begin{proof}
We are going to show that with these assumptions there is no non-trivial $P$-action on
$A_{K,\ell^n}(K_s)=(\bbZ\mod\ell^n)^{2g}$ for a suitable prime $\ell$ and every pro-$p$-group $P$.
We  do this by investigating the order of the corresponding automorphism groups. Following
lemma \ref{autord}, neither $\aut((\bbZ\mod\ell^n)^{2g})$ nor $\aut H$ for any subgroup
$H\sub(\bbZ\mod\ell^n)^{2g}$ do have elements of order $p$, if $p\nmid \ell^i-1$ for all
$i\in \set{1,\ldots, 2g}$. Without loss of generality, we can restrict ourselves to the group
$(\bbZ\mod\ell^n)^{2g}$. The last condition can be formulated as
\[
p\nmid \ord\aut((\bbZ\mod\ell^n)^{2g})\Iff \ell^i\not\equiv 1 \modulo p \quad\text{for all
$i\in\set{1,\dots,2g}$}.
\]
The group $(\bbZ\mod p)\units$ is  cyclic and of order $p-1$. With lemma \ref{dpnt}, we can choose a
prime $\ell$ which generates $(\bbZ\mod p)\units$. Then $p-1$ is the minimal exponent with the
property that
\[
\ell^{p-1}\equiv 1\modulo p.
\]
Therefore, it is minimal with $p\mid \ell^{p-1}-1$. Consequently, if
\[
2g<p-1,
\]
then $\aut((\bbZ\mod\ell^n)^{2g})$ does not have any elements of order $p$. As this inequality 
cannot be true for $p=2$, the above inequality leads to $2g+3\leq p$.
\par
With the semistable reduction theorem, \ref{ssred}, and the Galois criterion, \ref{galois_crit}, we
choose a finite Galois extension $L/K$ with Galois group $G'\defeq\gal(K_s/L)\sub
G\defeq\gal(K_s/K)$ and some $G'$-submodule $T'\sub T\defeq T_\ell(A_K)$ such that $G'$ acts
trivially on both $T'$ and $T\mod T'$.
\par
Now, we consider the tamely ramified extension $L^\tr/K$. We have seen in proposition \ref{gal_Ltr}
that the corresponding Galois group $\gal(K_s/L^\tr)$ is of the form $G'\cdot P$ for some
pro-$p$-group $P\sub G$. Now the computation (the limit varies over all open subgroups $P'\sub P$)
\[
\HH^0(P,(\bbZ\mod\ell^n)^{2g})=\indlim\HH^0(P\mod
P',((\bbZ\mod\ell^n)^{2g})^{P'})=\indlim(((\bbZ\mod\ell^n)^{2g})^{P'})=(\bbZ\mod\ell^n)^{2g}
\]
shows that every pro-$p$-group $P$ acts trivially on $T_\ell(A_K)$ (and, by lemma \ref{autord}, on
all its subgroups). Thus the product $G'\cdot P$ acts trivially on $T'$ and $T\mod T'$. Therefore,
$A_{L^\tr}$ has semistable reduction.
\end{proof}

We can infer numerous corollaries from this theorem. Amazingly, the threshold $2g+3\leq p$ for the
dimension $g=\dim A_K$ has many consequences for the groups of components and for the canonical
morphism $\oldphi_{A_K}\to\oldphi_{A_L}$. With the methods of this proof we can show:

\begin{prop}\label{prime_gal}
Let $A_K$ be a variety of dimension $g$ with $2g+3\leq p$. Let $L/K$ be a minimal Galois extension
with the property that $A_L$ reaches semistable reduction, then every prime divisor $q$ of $[L:K]$
is smaller than $2g+3$.
\end{prop}

\begin{proof}
We know that there exists a finite, tamely ramified extension $L/K$ such that $A_L$ reaches
semistable reduction. The extension $L/K$ induces an exact sequence
\[
0\longto\gal(K_s/L)\longto\gal(K_s/K)\Longto f\bbZ\mod n\longto 0,
\]
where $n\defeq[L:K]$ and where $\bbZ/n$ is isomorphic to the Galois group of $L/K$. Now, let $q\geq
2g+3$ be a prime. As in the proof of the theorem, we can choose a prime $\ell$ such that there is no
non-trivial action of a pro-$q$-group on $T_\ell(A_K)$ -- and on all of its subgroups. Let
$(\bbZ\mod n)_q$ be the $q$-part of $\bbZ/n$. Since $\bbZ/n$ is abelian, this corresponds to Galois
extensions $K\sub L'\sub L$. Then the preimage $f\inv((\bbZ\mod n)_q)=\gal(K_s/L)\cdot Q$ for some
pro-$q$-group $Q$ is the Galois group $\gal(K_s/L')$. Now, we can conclude as in the proof of the
theorem.
\end{proof}

Our main application of the above reduction theorem is the following result regarding
Grothendieck's pairing:

\begin{cor}\index{Grothendieck's pairing}
Let $R$ be a discrete valuation ring with perfect residue class field of characteristic $p$. If
$A_K$ is an abelian variety over $K\defeq\quot R$ of dimension $g$ with $2g+3\leq p$, then
Grothendieck's pairing for $A_K$ is perfect.
\end{cor}

\begin{proof}
Without loss of generality we can assume $R$ to be strictly henselian (\cite{sga7}, IX, 1.3.1).
Since $k$ is assumed to be perfect, it is algebraically closed. According to the above theorem,
$A_K$ acquires semistable reduction after a tamely ramified Galois extension of $L/K$. As
Gro\-then\-dieck's pairing is perfect for $A_L$, cf.\ \cite{We}, we can conclude by means of
proposition \ref{gp_perf_tamely} that Gro\-then\-dieck's pairing of $A_K$ is perfect.
\end{proof}

Following Poincar\'e's reducibility theorem (\cite{mum}, IV, 18, theorem 1)\index{Poincar\'e
reducibility theorem}, every abeli\-an variety is isogenous to a product of simple abelian
varieties and we can prove:

\begin{cor}
Let $R$ be as in the above corollary and let $A_K$ be isogenous to the product
$B_{K,1}\times\ldots\times B_{K,n}$ with $2\cdot\dim B_{K,i}+3\leq p$ for every $i\in\set{1,\ldots
n}$, then Grothendieck's pairing of $A_K$ is perfect.
\end{cor}

\begin{proof}
Each factor $B_{K,i}$ reaches semistable reduction after a tamely ramified extension $K_i$ of $K$.
If we take $K'$ to be the composite field of all $K_i$, then $K'/K$ is tamely ramified and the
product of the $B_{K,i}$ reaches semistable reduction over $K'$. Since $A_K$ and $\prod B_{K,i}$
are isogenous, $A_{K'}$ has semistable reduction if and only if each $B_{K',i}$ has semistable
reduction by \cite{blr}, 7.3, corollary 7. Now, proposition \ref{gp_perf_tamely} completes the
proof.
\end{proof}



\def\MakeUppercase#1{#1}



\end{document}